\documentclass[10pt,leqno]{amsart}
\usepackage{amscd}
\usepackage{color}
\usepackage{amssymb}
\usepackage{amsfonts}
\usepackage{latexsym}
\usepackage{verbatim}

\theoremstyle{plain}
\newtheorem{theorem}{Theorem}[section]

\newtheorem{lemma}[theorem]{Lemma}
\newtheorem{prop}[theorem]{Proposition}
\newtheorem{cor}[theorem]{Corollary}
\newtheorem{rem}[theorem]{Remark}

\renewcommand{\b}{\begin{equation}}
\newcommand{\e}{\end{equation}}

\newcommand\C{{\mathbb C}}
\newcommand\R{{\mathbb R}}

\newcommand{\ov}[1]{\overline{ #1}}

\newcommand{\N}{\nabla}

\newcommand{\F}{\mathcal{F}}
\begin{document}
\title[]{Some remarks on Calabi-Yau and Hyper-K\"ahler foliations}
\date{\today}
\author[G.~Habib]{Georges Habib}
\address{Lebanese University \\
Faculty of Sciences II \\
Department of Mathematics\\
P.O. Box 90656 Fanar-Matn \\
Lebanon}
\email[G.~Habib]{ghabib@ul.edu.lb}

\author[L. ~Vezzoni]{Luigi Vezzoni}
\address{Dipartimento di Matematica \\ Universit\`a di Torino\\
Via Carlo Alberto 10\\
10123 Torino\\ Italy} \email[L.~Vezzoni]{luigi.vezzoni@unito.it}
\thanks{This work was partially supported by  G.N.S.A.G.A. of I.N.d.A.M. and by the project FIRB \lq\lq Geometria Differenziale e Teoria Geometrica dell Funzioni\rq\rq\,. 
}
\begin{abstract}
We study Riemannian foliations whose transverse Levi-Civita connection $\nabla$ has special  holonomy. In particular, we focus on the case where ${\rm Hol}(\nabla)$ is contained either in ${\rm SU}(n)$ or in ${\rm Sp}(n)$.
We prove a Weitzenb\"{o}ck formula involving complex basic forms on K\"ahler foliations and we apply this formula for pointing out some properties of transverse Calabi-Yau structures. This allows us to prove that links provide examples of compact simply-connected contact Calabi-Yau manifolds.
Moreover, we show that a simply-connected compact manifold with a K\"ahler foliation admits a transverse hyper-K\"ahler structure if and only if it admits a compatible transverse hyper-Hermitian structure. This latter result is  the  \lq \lq foliated version'' of a theorem proved by Verbitsky in \cite{verb}. In the last part of the paper we adapt our results to the Sasakian case, showing in addition that a compact Sasakian manifold has trivial transverse holonomy if and only if it is a compact quotient of the  Heisenberg Lie group.
\end{abstract}

\maketitle
\tableofcontents

\section{Introduction}
Riemannian foliations were introduced by B. Reinhart  in \cite{Rein} and are a natural generalization of Riemannian submersions. Roughly speaking, a {\em Riemannian foliation} on a manifold $M$ is a decomposition of $M$ into submanifolds given by local Riemannian submersions to a base Riemannian manifold $T$ whose metric is invariant by the transition maps.
Riemannian foliations are characterized by the existence of a Riemannian metric on the whole manifold whose restriction to the normal bundle depends only on the transverse variables of a local chart. One of the basic tool for studying the geometry of Riemannian foliations is the holonomy group of the so-called {\em transverse Levi-Civita connection} $\N$. This connection is defined as the pull-back of the Levi-Civita connection of the base manifold by the local submersions. Many additional structures on a foliated manifold can be described in terms of the holonomy group of $\N$. For instance,  {\em transverse K\"ahler structures} are defined as Riemannian foliations having the holonomy group of $\N$ contained in
${\rm U}(n)$, where $q=2n$ is the codimension of the foliation. K\"ahler foliations play an important role in many different geometrical contexts: for instance,
Sasakian structures  and Vaisman metrics induce a K\"ahler foliation.

In this paper, we investigate the geometry of Riemannian foliations having special transverse holonomy. In particular, we focus on the case of a foliated manifold having either  ${\rm Hol}(\N)\subseteq {\rm SU}(n)$ or ${\rm Hol}(\N)\subseteq {\rm Sp}(n)$. The case ${\rm Hol}(\N)\subseteq {\rm SU}(n)$ corresponds to the geometry of {\em Calabi-Yau} foliations, while ${\rm Hol}(\N)\subseteq {\rm Sp}(n)$ to {\em hyper-K\"ahler } foliations. Besides other reasons, our study is motivated by the El Kacimi paper \cite{EKA}  containing the foliated version of the Calabi-Yau theorem. Examples of Calabi-Yau foliations are provided by submersions over Calabi-Yau manifolds, desingularizations of Calabi-Yau orbifolds and contact Calabi-Yau structures (see \cite{TV}); while examples of hyper-K\"ahler foliations can be obtained by considering submersions over hyper-K\"ahler manifolds, desingularizations of   hyper-K\"ahler orbifolds, $3$-cosymplectic structures (see e.g. \cite[Section 13.1]{BGlibro}) and by the connected sum of some copies of $S^2\times S^3$ (see \cite{4,cuadros} and the last paragraph of the present paper).

\medskip
As a first result of the paper, we provide a  Weitzenb\"ock formula for K\"ahler foliations
(see theorem \ref{W} in section  \ref{Weitzenbock}).  This formula allows us to establish some analogies between foliated Calabi-Yau manifolds and classical Calabi-Yau manifolds (see section \ref{sectionSU(n)}). As main result about hyper-K\"ahler foliations, we prove that every simply-connected compact manifold carrying a K\"ahler foliation and a compatible transverse hyper-complex structure actually admits a hyper-K\"ahler foliation. That is the foliated version of a theorem of Verbitsky (see \cite{verb}). A key ingredient in the proof of this last result is the existence and uniqueness of a special connection having skew-symmetric transverse torsion on every manifold carrying a Hermitian foliation.  The existence of this connection in contact metric manifolds was showed by  Friedrich and Ivanov in \cite{FI}. In the last part of the paper we consider Sasakian manifolds. We prove that a transversally flat compact Sasakian manifold is always a compact quotient of the Heisenberg group (see theorem \ref{1}); we point out that some links provide examples of compact simply-connected
contact Calabi-Yau manifolds and we adapt some results proved in the first part to the Sasakian case.

\medskip

\bigbreak\noindent{\it Acknowledgments.} The authors are grateful to Charles P. Boyer, Thomas Brun Madsen,
Gueo Grancharov,  Valentino Tosatti and Misha Verbitsky for useful conversations, suggestions and remarks.

\section{Preliminaries}
In this section, we recall some basic materials about foliations, transverse structures and basic cohomology; we refer to \cite{T, molino, BGlibro} and the references therein for detailed expositions about these topics.
\subsection{Transverse structures on foliations}
Let $M$ be a smooth manifold. A foliation $\mathcal F$ on $M$ of codimension $q$ can be defined as an open cover $\{U_k\} $
of $M$ together with a family of submersions $f_k\colon U_k\to T$ over a manifold $T$ of dimension $q$ (called the base of the foliation) such that whenever $U_j\cap U_k\neq \emptyset$ there exists a diffeomorphism $\gamma_{jk}\colon f_j( U_j\cap U_k)\to f_k(U_j\cap U_k)$ such that
$$
f_{j}=\gamma_{jk}\circ f_k\,.
$$
The basic examples of foliations are provided by global submersions.  A {\em transverse structure} on a foliated manifold $(M,\F)$ is by definition a geometric structure on the base manifold $T$ which is invariant by the transition maps $\gamma_{jk}$. A foliation is called {\em Riemannian} if it admits a transverse Riemannian metric $g_Q$. In contrast to the non-foliated case the existence of a transverse metric is not always guaranteed.  
 Given a foliated manifold $(M,\F)$, we denote by $L$ the subbundle of $TM$ induced by $\F$; by $Q$ the normal bundle $TM/L$; by $\pi\colon TM\to Q$ the natural projection and for a section $X$ of $TM$ we usually set $X_Q:=\pi(X)$. 
A transverse metric on $\F$ induces a metric $g_Q$  along the fibers of $Q$  which always satisfies the  so-called {\em holonomy invariant} condition
\begin{equation}\label{bundle-like}
\mathcal{L}_Xg_Q=0
\end{equation}
for every $X\in \Gamma(L)$, where $\mathcal{L}$ denotes the Lie derivative. By using the projection $\pi$ we can regard $g_Q$ as a symmetric tensor on $M$
which can be always ``completed'' to a global metric $g$ on $M$. This means that there exists a Riemannian metric $g$ on $M$ such that
$$
g(X,Y)=g_Q(X_Q,Y_Q)
$$
for every sections $X,Y$ of $L^{\perp}$. Such a global metric $g$ can be explicitly defined by starting from an arbitrary Riemannian metric $g'$ of $M$ and by setting
$$
g(X,Y):=g'(X_L,Y_L)+g_Q(X_Q,Y_Q)
$$
for every $X,Y\in \Gamma(TM)$, where the subscript $L$ here denotes the projection onto $\Gamma(L)$ with respect to $g'$.
On the other hand, given a foliated Riemannian manifold $(M,\F,g)$, the metric $g$ always induces
a metric $g_Q$ along the fibers of $Q:=TM/L\simeq L^{\perp}$. Such $g_Q$ makes $\F$ a Riemannian foliation if and only if it satisfies the {\em holonomy invariant} condition \eqref{bundle-like}.
In this case $g$ is called a {\em bundle-like} metric.

\medskip
Given a Riemannian foliation  $(\F,g_Q)$ on  a manifold $M$, some special transverse structures can be characterized in terms of the holonomy of the so-called {\em transversal Levi-Civita connection}. This connection is defined  as the unique connection $\nabla$ on $Q$ preserving $g_Q$ and having vanishing transverse torsion, i.e. 
$$
\nabla_{X}Y_Q-\nabla_{X}Y_Q=[X,Y]_Q
$$
for all $X,Y\in \Gamma(TM)$. The connection 
 $\N$ can be defined in an explicit way in terms of the Levi-Civita connection $\nabla^g$ of a bundle-like metric $g$ inducing $g_Q $ by setting
\b
\nabla_{X}s=
\begin{cases}
\begin{array}{ccl}
[X,\sigma(s)]_{Q}& \mbox{if}& X\in \Gamma(L),\\
(\N^g_{X}\sigma(s))_{Q}& \mbox{if} &X\in \sigma(Q),
\end{array}
\end{cases}
\e
for every  $s\in \Gamma(Q)$, where  $\sigma\colon \Gamma(Q)\to \Gamma(L^{\perp})$ is the natural isomorphism (see e.g. \cite{T}). This last description of $\nabla$ does not depend on the choice of the metric $g$. The curvature $R^\nabla$ of $\nabla$ vanishes along the leaves of $\F$ (see \cite{T}): that is, $X\lrcorner R^\nabla=0$ for all $X\in \Gamma(L)$. That allows us to define the {\em transverse Ricci tensor} as  
$$
{\rm Ric^{\nabla}}(s_1,s_2)=\sum_{k=1}^qg_{Q}(R^{\nabla}(\tilde e_k,\tilde s_1)s_2,e_k)
$$
for every $s_1,s_2\in \Gamma(Q)$,
where $\{e_k\}$ is an arbitrary orthonormal frame of $\Gamma(Q)$ and $\tilde e_k$ and $\tilde s_1$ are arbitrary vector fields on $M$ projecting onto $e_k$ and $s_1$, respectively.  

\medskip
Now we recall the definition of the {\em basic cohomology complex.}
A differential form $\alpha$ on a foliated manifold $(M,\F)$ is called {\em basic} if it is constant along the leaves of $\F$, i.e. if it satisfies
$$
X\lrcorner \alpha=0\,,\quad X\lrcorner d\alpha=0\,,
$$
where $X\lrcorner $ denotes the contraction along  $X\in \Gamma(L)$.
It is straightforward to see that the exterior derivative $d$ preserves the set of basic forms $\Lambda_B(M)$ and its restriction $d_B$ to this set is used to define the so-called basic cohomology groups $H^r_B(M)$ (see e.g. \cite{molino, T}). From now until the end of this paragraph, we assume that $M$ compact and $\F$ is transversally oriented. 

\medskip
The orientation of $\F$ induces a {\em basic Hodge star operator}
$$
*_B\colon \Lambda_B^r(M)\to \Lambda_B^{q-r}(M)
$$
in the usual way. Moreover, the so-called {\em characteristic form} $\chi_\F$ is defined as the volume form of the leaves and is locally given by 
$$
\chi_\F(X_1,\dots,X_p)={\rm det}\left( g(X_i,E_j)\right)\,,\mbox{ for } X_r\in \Gamma(TM)
$$
where $\{E_k\}_{k=1,\cdots,p}$ is a local oriented orthonormal frame of $\Gamma(L)$ and $p$ is the rank of $L$. Thus, one may define a natural scalar product on the set of basic forms by setting
\begin{equation}\label{scalarbasic}
(\alpha,\beta)_B:=\int_M\alpha\wedge *_B\beta\wedge\chi_\F
\end{equation}
Let $\delta_B$ be the formal adjoint of $d_B$ with respect to the scalar product \eqref{scalarbasic} and $\Delta_B:=d_B\delta_B+\delta_Bd_B$ the {\em basic Laplacian operator}. Then $\Delta_B$ is a transversally elliptic operator and by \cite{EHS} its kernel has finite dimension. Moreover, in view of the basic Hodge theorem, $H^r_B(M)$ is isomorphic to $\mathcal{H}^r_B(M):=\ker \Delta_B\cap \Lambda_B^r(M)$ \cite{EK-Hc2,K-To15}. Therefore the basic cohomology groups have finite dimensions, but in general they do not satisfy the Poincar\'e duality (see \cite{Ca} for an example). 
This latter is guaranteed when the foliation is {\em taut}, i.e. when the leaves are minimal with respect to a suitable bundle-like metric or, equivalently in view of \cite{Ma}, when the top basic cohomology group $H^{q}_B(M)$ is non-trivial (and then it is $1$-dimensional). 

\medskip 
Given a foliated manifold with a bundle-like metric $(M,\F,g),$ the
{\em mean curvature vector field} is defined as
$$
H=\sum_{l=1}^p (\nabla^g_{E_l} E_l)_Q
$$
where $\{E_l\}_{l=1,\cdots,p}$ is an arbitrary orthonormal frame of $\Gamma(L)$. The $1$-form $\kappa$ dual  to  $H$ is usually called {\em the mean curvature form}. Notice that the leaves of $\F$ are minimal with respect to $g$ if and only if $\kappa$ vanishes. In view of \cite{domin}, it is  always possible to find a compatible bundle-like metric $g$ on $M$ whose induced $\kappa$ is basic.  Moreover, when $\kappa$ is basic it is automatically closed
(see \cite{Kamber,T}) and consequently on a compact simply-connected manifold  every Riemannian foliation is taut (see \cite{T}). Moreover, 
the forms $\kappa$ and $\chi_\F$ are related by the following formula arising from \cite{Ru}
\begin{equation}\label{Ru}
 \alpha\wedge d\chi_\F= -\alpha\wedge \kappa \wedge \chi_\F
\end{equation}
holding for every basic form $\alpha$ of degree $q-1$. Hence if $(\F,g_Q)$ is an orientable taut Riemannian
foliation there exists a $p$-form $\chi_\F$ which restricts to a volume along the leaves and satisfies 
$$
\alpha\wedge d\chi_\F=0
$$
for every basic $(q-1)$-form $\alpha$. 

\subsection{Transverse K\"ahler structures}
In this section, we focus on transverse Hermitian and transverse K\"ahler structures.

Let $(M,\F,g_Q)$ be a manifold equipped with a Riemannian foliation. A {\em transverse complex} structure on $(M,\F)$ is an endomorphism $J$ of $Q$ such that
$$
J^2=-{\rm Id}_Q\,,
$$
(i.e. a transverse almost complex structure).
The pair $(g_Q,J)$ is said to be a {\em transverse Hermitian} structure if and only if
$$
g_Q(J\cdot,J\cdot)=g_Q(\cdot,\cdot)\,.
$$
In this case, the triple $(\F,g_Q,J)$ is called a {\em Hermitian foliation}. The basic $2$-form $\omega$ obtained as the pull-back to $M$ of the skew-symmetric tensor $g_Q(J\cdot,\cdot)$ is usually called the {\em fundamental form} of the foliation and it is closed if and only if $(g_Q,J)$ is induced by a transverse K\"ahler structure. In analogy to the non-foliated case, the condition $d\omega=0$ writes in terms of the transverse Levi-Civita connection as $\nabla J=0$. 

\medskip 
Given a transverse complex structure $J$ on a foliated manifold $(M,\mathcal{F})$, the complexified normal bundle $Q^{\C}:=Q\otimes \C$ splits into the two eigenbundles $Q^{1,0}$ and $Q^{0,1}$ corresponding to the eigenvalues $i$ and $-i$ of $J$ and we have
$$
\Lambda^r(Q^*)\otimes \C=\mathop\bigoplus\limits_{i+j=r}\Lambda^{i,j}(Q)\,.
$$
 Since the space of smooth sections of $\Lambda^r(Q^*)$ is isomorphic to the space of $r$-forms $\alpha$ on $M$ satisfying $X\lrcorner \alpha=0$ for every $X\in \Gamma(L)$, $J$ induces an operator (which we still denote by $J$) on complex basic forms. 
Therefore we have the natural splitting 
$$
\Lambda_B^r(M,\C)=\mathop\bigoplus\limits_{i+j=r}\Lambda_B^{i,j}(M),
$$
and the complex extension of $d_B$ splits as 
$$
d_B=\partial_B+\bar\partial_B
$$
where 
$$
\partial _{B}\colon \Lambda_{B}^{i,j}(M)\rightarrow \Lambda _{B}^{i+1,j}(M)\,,\quad \bar{\partial}_{B}:\Lambda _{B}^{i,j}(M)\rightarrow \Lambda _{B}^{i,j+1}(M)\,.
$$
In analogy to the non-foliated case, we have $\partial _{B}^2=\bar{\partial}_{B}^2=0$, $\partial _{B}\bar{\partial}_{B}+\bar{\partial}_{B}\partial _{B}=0.$ 

Moreover, if $(g_Q,J)$ is a transverse K\"ahler structure, the transverse Ricci tensor of $g_Q$ is $J$-invariant and thus induces the {\em transverse Ricci form} $\rho_B$. This latter is the closed basic form obtained as the pull-back of $\frac{1}{2\pi}{\rm Ric}^{\nabla}(J\cdot,\cdot)$ to $\Lambda^2_B(M)$. According to the non-foliated case, one has  
$$
\rho_B(\cdot,\cdot)=-\frac{i}{2\pi}\partial_B\bar\partial_B {\rm log}(G)
$$
where $G={\rm det}(g_{r\bar s})$ and the functions $g_{r\bar s}$ are computed with respect to suitable transverse complex coordinates. The class of $\rho_B$ in $H_B^2(M,\R)$ is by definition the {\em first basic Chern class} of $(\F,J)$ and it is usually denoted by $c_B^1$.
The following important theorem is due to El Kacimi and provides a foliated version of the celebrated Calabi-Yau theorem:

\begin{theorem}[\cite{EKA}]\label{transvcalabiyau} Let $(M,\mathcal{F},g_Q,J$) be a compact manifold endowed with a taut K\"ahler foliation and let $\omega$ be its fundamental form. If $c_B^1$
is represented by a real basic $(1,1)$-form $\rho'_B$, then $\rho'_B$ is the basic Ricci
 form of a unique transverse K\"ahler form $\omega'$
in the same basic cohomology class
of $\omega$. In particular, if $c_B^1=0$, then there exists a  transverse K\"ahler metric having vanishing transverse Ricci tensor.
\end{theorem}

\begin{rem}\label{remarkCY}
{\em Let $(M,\F,g_Q,J)$ be a Hermitian foliation of real codimension $2n$. Then the first Chern class of $K:=\Lambda^{n,0}(Q)$ vanishes if and only if there exists a nowhere vanishing $\eta\in\Lambda^{n,0}(Q)$. Such an $\eta$ induces a nowhere vanishing section $\psi$ of $\Lambda^{2n}(M,\C)$ satisfying $X\lrcorner \psi=0 $ for every section $X$ of $L$.  Since  $\psi$ is not necessary basic, condition $c^1(K)=0$ does not imply $c^1_B(\F)=0$ and it is quite natural asking if the existence of a nowhere vanishing basic $(n,0)$-form $\psi$ implies $c^1_B(\F)=0$. This fact is certainly true if $M$ is simply-connected since in this case the same argument as in the non foliated case  (see e.g. \cite{Audin}) allows us to prove that 
$\rho=-i\partial_B\bar \partial_B f$, where $f=g_Q(\eta,\bar \eta)$ is the  pointwise norm of the form $\eta\in \Lambda_B^{n,0}(M)$ corresponding to $\psi$.}  
\end{rem}

An important class of transverse K\"ahler structures is provided by Sasakian manifolds \cite{einstein1}. These latters are characterized by the existence of a unit Killing vector field $\xi$ on a $(2n+1)$-dimensional Riemannian manifold $(M,g)$ such that the tensor field $\Phi$ defined for $X\in \Gamma(TM)$ as $\Phi(X)=\nabla^g_X\xi$ satisfies
\begin{enumerate}
\item[1.] $\Phi^2=-{\rm Id}_{TM}+\xi^\flat\otimes \xi,$

\vspace{0.1cm}
\item[2.] $(\nabla^g_X \Phi)(Y)=g(\xi,Y)X-g(X,Y)\xi,$
\end{enumerate}
where $X, Y$ are vector fields in $\Gamma(TM).$
The vector field $\xi$ is called the {\em Reeb vector field} and generates a $1$-dimensional Riemannian foliation $\mathcal{F}$ such that the restriction of $\Phi$ to $\xi^\perp$ gives a transverse K\"ahler structure with vanishing mean curvature. Usually a Sasakian structure is denoted by a quadruple $(\xi,\eta,\Phi,g)$, where $\eta=\xi^\flat$ is the $1$-form dual to $\xi$.  About the geometry of Sasakian manifolds we refer to \cite{Sas60,SH62, BGlibro} and the references therein, whilst for the Sasakian-version of theorem \ref{transvcalabiyau} we refer to \cite{BGM} and \cite{BGN}.

\medskip
Many interesting examples of non-K\"ahlerian complex spaces carry a K\"ahler foliation. For instance, Vaisman manifolds, Calabi-Eckmann manifolds and some  Oeljeklaus-Toma manifolds carry K\"ahler foliations having the transverse K\"ahler form exact (see e.g. \cite{dragomir, ornea} and the references therein).

\medskip
Another interesting examples of K\"ahler foliations come from the physical study of A-branes (see \cite{K,HO}) and they are obtained by considering a coisotropic submanifold of a K\"ahler manifold.
A submanifold of $N\hookrightarrow (M,\omega,J,g)$ of a K\"ahler manifold 
is called {\em coisotropic} if 
\begin{equation}\label{coso}
TN^{\omega}\subset TN
\end{equation}
where
$$
\left(T_{y}N\right)^{\omega}=\{v\in  T_yM\,\,:\,\, \omega(v,w)=0\,\,,\mbox{ for all }w\in T_yN \}\,.
$$
The coisotropic condition \eqref{coso} implies that $\F_y:=(T_yN)^{\omega}$ is a distribution on $N$ which is intregrable by the closure of  $\omega$. Moreover, the complex structure $J$ always preserves  the orthogonal complement
of the bundle $L$ induced by $\F$ and we have the following
\begin{prop}\label{coisotropic}
The metric $g$ is always bundle-like with respect to $\F$ and $\F$ is a locally trivial K\"ahler  foliation.
\end{prop}
\begin{proof}
The metric $g$ is bundle-like if and only if
$$
\mathcal{L}_{X}g(X_1,X_2)=0
$$
for every $X\in \Gamma(L)$ and $X_1,X_2\in \Gamma(L^\perp)$. Such a relation can be read  in terms of
the Levi-Civita connection $\nabla^g$ of $g$ as
$$
g(\nabla^g_{X_1}X,X_2)=-g(\nabla^g_{X_2}X,X_1)\,.
$$
Now it is enough to observe that the K\"ahler condition $\nabla^g\omega=0$ implies
$$
g(\nabla^g_{X_1}X,X_2)=0
$$
since $g(\nabla^g_{X_1}X,X_2)=\omega(\nabla^g_{X_1}X,JX_2)$ and
$$
0=(\nabla^g_{X_1}\omega)(X,JX_2)=X_1\omega(X,JX_2)-\omega(\nabla^g_{X_1}X,JX_2)-\omega(X,\nabla^g_{X_1}JX_2)=
-\omega(\nabla^g_{X_1}X,JX_2)\,.
$$
\end{proof}

\section{A Weitzenb\"{o}ck formula for transverse K\"ahler structures}\label{Weitzenbock}

In this section, we establish a transverse Weitzenb\"{o}ck formula for complex-valued basic forms on K\"ahler foliations and we also derive some vanishing results. We strictly follow the classical computations in the non-foliated case as described in \cite{moroianu}.

\medskip
Let $(M,\mathcal{F},g_Q,J)$ be a compact manifold  with a K\"ahler foliation of codimension $q$ and let $g$ a bundle-like metric on $M$ inducing $g_Q$. We may assume in view of \cite{domin} that the mean curvature form $\kappa$ of the foliation is a basic (otherwise we can work with the basic component of $\kappa$). 
From now until the end of this section, we identify the bundle $Q$ with its dual $Q^*$ and vectors  with  $1$-forms by using the transverse metric.  
In analogy to the non-foliated case, the two operators $\partial _{B}$ and $\bar{\partial}_{B}$ can be written in terms of the transverse Levi-Civita connection $\nabla$ of $(\F,g_Q)$ as
$$
\partial_{B}=\frac{1}{2}\sum_{j=1}^q(e_{j}+iJe_{j})\wedge \nabla_{e_{j}}\,\,\text{and}\,\,\bar{\partial}_{B}=\frac{1}{2}\sum_{j=1}^q(e_{j}-iJe_{j})\wedge \nabla_{e_{j}},
$$
where $\{e_j\}_{j=1,\cdots,q}$ is a local orthonormal frame of $\Gamma(Q)$.  
 In \cite{HabRi}, the authors established a Bochner-Weitzenb\"ock formula for the Laplacian operator corresponding to the {\em twisted derivative} $\tilde{d}_B=d_B-\frac{1}{2}\kappa\wedge$. As mentionned before the mean curvature is $d_B$-closed; in particular we have that $\tilde{d}_B^2=0$. In the same spirit as \cite{HabRi}, we modify the operators $\partial_{B}$ and $\bar{\partial}_{B}$ by introducing the following two twisted operators
\begin{equation*}
\tilde{\partial}_{B}=\frac{1}{2}\sum_{j=1}^q(e_{j}+iJe_{j})\wedge \nabla
_{e_{j}}-\frac{1}{4}(\kappa +iJ\kappa )\wedge
\,\,\text{and}\,\,
\tilde{\bar{\partial}}_{B}=\frac{1}{2}\sum_{j=1}^q(e_{j}-iJe_{j})\wedge \nabla
_{e_{j}}-\frac{1}{4}(\kappa -iJ\kappa )\wedge.
\end{equation*}
We readily have that $\tilde{d}_{B}=\tilde{\partial}_{B}+\tilde{\bar{\partial}}_{B}$ and  the formal adjoint  to $\tilde{\bar{\partial}}_{B}$ with respect to the scalar product \eqref{scalarbasic} writes as
\begin{equation*}
\left(\tilde{\bar{\partial}}_{B}\right)^{\ast }=-\frac{1}{2}\sum_{j=1}^q(e_{j}+iJe_{j})
\lrcorner \nabla _{e_{j}}+\frac{1}{4}(\kappa +iJ\kappa )\lrcorner\,.
\end{equation*}
Now we state the main result of this section:
\begin{theorem}\label{W} Let $(M,g,\mathcal{F},J)$ be a compact Riemannian manifold endowed with a K\"ahler foliation. Assume that the mean curvature $\kappa$ is basic-harmonic. Then the following Weitzenb\"{o}ck-type formula  holds
$$
\begin{aligned}
2((\tilde{\bar{\partial}}_{B})^{\ast }\tilde{\bar{\partial}}_{B}+\tilde{\bar{\partial}}_{B}(\tilde{\bar{\partial}}_{B})^{\ast }) =&\,\nabla ^{\ast }\nabla
+\frac{1}{4}|\kappa |^{2}+\mathfrak{R}+\frac{1}{4}(e_{j}-iJe_{j})\wedge(\nabla _{e_{j}}\kappa +iJ(\nabla _{e_{j}}\kappa ))\lrcorner \\
&\,+\frac{i}{2}(Je_{\ell },\nabla _{e_{\ell }}\kappa )-\frac{1}{4}(\nabla_{e_{\ell }}\kappa -iJ(\nabla _{e_{\ell }}\kappa ))\wedge (e_{\ell}+iJe_{\ell }\lrcorner )\,,
\end{aligned}
$$
where the third term $\mathfrak{R}$ is
\begin{equation*}
\mathfrak{R}=\frac{i}{2}R^\nabla(Je_{j},e_{j})-\frac{1}{2}(e_{j}-iJe_{j})\wedge
(e_{\ell }+iJe_{\ell })\lrcorner R^\nabla(e_{j},e_{\ell }).
\end{equation*}
\end{theorem}
\begin{proof}
Let $p$ be a fixed point of $M$ and  let $\{e_i\}_{i=1,\cdots,q}$ be an orthonormal frame of $\Gamma(Q)$ which we may assume to be parallel at $p$. Then working at $p$ we have

$$
\begin{aligned}
2(\tilde{\bar{\partial}}_{B})^{\ast}\tilde{\bar{\partial}}_{B} =\,&(\tilde{
\bar{\partial}}_{B})^{\ast}\left( \sum_{j=1}^q(e_{j}-iJe_{j})\wedge \nabla
_{e_{j}}-\frac{1}{2}(\kappa -iJ\kappa )\wedge \right)\\
=&-\frac{1}{2}\sum_{\ell=1}^q(e_{\ell }+iJe_{\ell })\lrcorner \nabla _{e_{\ell
}}\left( \sum_{j=1}^q(e_{j}-iJe_{j})\wedge \nabla _{e_{j}}-\frac{1}{2}(\kappa
-iJ\kappa )\wedge \right)\\
&+\frac{1}{4}(\kappa +iJ\kappa )\lrcorner \left(
\sum_{j=1}^q(e_{j}-iJe_{j})\wedge \nabla _{e_{j}}-\frac{1}{2}(\kappa -iJ\kappa
)\wedge \right)
\end{aligned}
$$
which yields to
$$
\begin{aligned}
2(\tilde{\bar{\partial}}_{B})^{\ast }\tilde{\bar{\partial}}_{B} =&-\frac{1}{2}\sum_{\ell=1}^q(e_{\ell }+iJe_{\ell })\lrcorner \left(
\sum_{j=1}^q(e_{j}-iJe_{j})\wedge \nabla _{e_{\ell }}\nabla _{e_{j}}-\frac{1}{2}
(\nabla _{e_{\ell }}\kappa -iJ(\nabla _{e_{\ell }}\kappa ))\wedge \,-\frac{1
}{2}(\kappa -iJ\kappa )\wedge \nabla _{e_{\ell }}\right)   \notag \\
&+\frac{1}{4}(\kappa +iJ\kappa )\lrcorner \left(
\sum_{j=1}^q(e_{j}-iJe_{j})\wedge \nabla _{e_{j}}-\frac{1}{2}(\kappa -iJ\kappa
)\wedge \right).
\end{aligned}
$$
Then we get
$$
\begin{aligned}
2(\tilde{\bar{\partial}}_{B})^{\ast}\tilde{\bar{\partial}}_{B} =&-\frac{1}{
2}(2\delta ^{\ell j}-2ig(Je_{j},e_{\ell }))\nabla _{e_{\ell }}\nabla _{e_{j}}+
\frac{1}{2}(e_{j}-iJe_{j})\wedge (e_{\ell }+iJe_{\ell
})\lrcorner \nabla _{e_{\ell }}\nabla _{e_{j}}  \notag \\
&+\frac{1}{2}g(\nabla _{e_{\ell }}\kappa ,e_{\ell })+\frac{i}{2}(Je_{\ell
},\nabla _{e_{\ell }}\kappa )-\frac{1}{4}(\nabla _{e_{\ell }}\kappa
-iJ(\nabla _{e_{\ell }}\kappa ))\wedge (e_{\ell }+iJe_{\ell })\lrcorner  \\
&+\frac{1}{2}g(\kappa ,e_{\ell })\nabla _{e_{\ell }}-\frac{i}{2}g(J\kappa
,e_{\ell })\nabla _{e_{\ell }}-\frac{1}{4}(\kappa -iJ\kappa )\wedge (e_{\ell
}+iJe_{\ell })\lrcorner \nabla _{e_{\ell }}  \notag \\
&+\frac{1}{2}g(\kappa ,e_{j})\nabla _{e_{j}}+\frac{i}{2}g(J\kappa
,e^{j})\nabla _{e_{j}}-\frac{1}{4}(e_{j}-iJe_{j})\wedge (\kappa +iJ\kappa
)\lrcorner \nabla _{e_{j}}  \notag \\
&-\frac{1}{4}|\kappa |^{2}+\frac{1}{8}(\kappa -iJ\kappa )\wedge (\kappa
+iJ\kappa )\lrcorner
\end{aligned}
$$
which gives that
$$
\begin{aligned}
2(\tilde{\bar{\partial}}_{B})^{\ast}\tilde{\bar{\partial}}_{B} =&\nabla ^{\ast }\nabla +ig(Je_{j},e_{\ell })\nabla _{e_{\ell }}\nabla
_{e_{j}}+\frac{1}{2}(e_{j}-iJe_{j})\wedge (e_{\ell }+iJe_{\ell
})\lrcorner \nabla _{e_{\ell }}\nabla _{e_{j}}  \notag \\
&+\frac{1}{2}\mathrm{div}_Q(\kappa )+\frac{i}{2}g(Je_{\ell },\nabla _{e_{\ell
}}\kappa )-\frac{1}{4}(\nabla _{e_{\ell }}\kappa -iJ(\nabla _{e_{\ell
}}\kappa ))\wedge (e_{\ell }+iJe_{\ell })\lrcorner   \notag \\
&-\frac{1}{4}(\kappa -iJ\kappa )\wedge (e_{\ell
}+iJe_{\ell })\lrcorner \nabla _{e_{\ell }}-\frac{1}{4}(e_{j}-iJe_{j})\wedge
(\kappa +iJ\kappa )\lrcorner \nabla _{e_{j}}  \notag \\
&-\frac{1}{4}|\kappa |^{2}+\frac{1}{8}(\kappa -iJ(\kappa ))\wedge (\kappa
+iJ(\kappa)),
\end{aligned}
$$
(here and in the following we omit the symbol of sum).
In the last equality above, we have made use of the relation
$$
\nabla^*\nabla=-\sum_{j=1}^q\nabla_{e_j}\nabla_{e_j}+\nabla_{\kappa}.
$$
Since we are assuming that the mean curvature form is basic-harmonic, the divergence of $\kappa$ is thus equal to the square of its norm. Thus, we have
$$
\begin{aligned}
2(\tilde{\bar{\partial}}_{B})^{\ast }\tilde{\bar{\partial}}_{B} =&\nabla ^{\ast }\nabla +\frac{1}{4}|\kappa |^{2}+\frac{i}{2}
g(Je_{j},e_{\ell })R^\nabla(e_{\ell },e_{j})+\frac{1}{2}(e_{j}-iJe_{j})\wedge \left( (e_{\ell }+iJe_{\ell })\lrcorner \nabla
_{e_{\ell }}\nabla _{e_{j}}\right) \\
&+\frac{i}{2}g(Je_{\ell },\nabla _{e_{\ell }}\kappa )-\frac{1}{4}(\nabla
_{e_{\ell }}\kappa -iJ(\nabla _{e_{\ell }}\kappa ))\wedge (e_{\ell
}+iJe_{\ell })\lrcorner \\
&-\frac{1}{4}(\kappa -iJ\kappa )\wedge (e_{\ell }+iJe_{\ell })\lrcorner
\nabla _{e_{\ell }}-\frac{1}{4}(e_{j}-iJe_{j})\wedge (\kappa +iJ\kappa
)\lrcorner \nabla _{e_{j}}+\frac{1}{8}(\kappa -iJ\kappa )\wedge (\kappa
+iJ\kappa )\lrcorner \\
=&\nabla ^{\ast }\nabla +\frac{1}{4}|\kappa |^{2}+\frac{i}{2}
R^\nabla(Je_{j},e_{j})+\frac{1}{2}(e_{j}-iJe_{j})\wedge (e_{\ell }+iJe_{\ell
})\lrcorner \nabla _{e_{j}}\nabla _{e_{\ell }} \\
&+\frac{1}{2}(e_{j}-iJe_{j})\wedge (e_{\ell }+iJe_{\ell })\lrcorner
R^\nabla(e_{\ell },e_{j})+\frac{i}{2}g(Je_{\ell },\nabla _{e_{\ell }}\kappa )-\frac{1
}{4}(\nabla _{e_{\ell }}\kappa -iJ(\nabla _{e_{\ell }}\kappa ))\wedge
(e_{\ell }+iJe_{\ell })\lrcorner \\
&-\frac{1}{4}(\kappa -iJ\kappa )\wedge (e_{\ell }+iJe_{\ell })\lrcorner
\nabla _{e_{\ell }}-\frac{1}{4}(e_{j}-iJe_{j})\wedge (\kappa +iJ\kappa
)\lrcorner \nabla _{e_{j}}+\frac{1}{8}(\kappa -iJ\kappa )\wedge (\kappa
+iJ\kappa )\lrcorner,
\end{aligned}
$$
which finally gives
$$
\begin{aligned}
2(\tilde{\bar{\partial}}_{B})^{\ast }\tilde{\bar{\partial}}_{B}
=&\nabla ^{\ast }\nabla +\frac{1}{4}|\kappa |^{2}+\mathfrak{R}+\frac{1}{2}%
(e_{j}-iJe_{j})\wedge (e_{\ell }+iJe_{\ell })\lrcorner \nabla _{e_{j}}\nabla
_{e_{\ell }}+\frac{i}{2}(Je_{\ell },\nabla _{e_{\ell }}\kappa ) \\
&-\frac{1}{4}(\nabla _{e_{\ell }}\kappa -iJ(\nabla _{e_{\ell }}\kappa
)\wedge (e_{\ell }+iJe_{\ell })\lrcorner -\frac{1}{4}(\kappa -iJ\kappa
)\wedge (e_{\ell }+iJe_{\ell })\lrcorner \nabla _{e_{\ell }}\\
&-\frac{1}{4}
(e_{j}-iJe_{j})\wedge (\kappa +iJ\kappa )\lrcorner \nabla _{e_{j}} 
+\frac{1}{8}(\kappa -iJ\kappa )\wedge (\kappa +iJ\kappa )\lrcorner\,.
\end{aligned}
$$
On the other hand,
$$
\begin{aligned}
2\tilde{\bar{\partial}}_{B}(\tilde{\bar{\partial}}_{B})^{\ast } =&\tilde{
\bar{\partial}}_{B}(-(e_{\ell }+iJe_{\ell })\lrcorner \nabla _{e_{\ell }}+
\frac{1}{2}(\kappa +iJ\kappa )\lrcorner ) \\
=&\frac{1}{2}\{(e_{j}-iJe_{j})\wedge \nabla _{e_{j}}(-(e_{\ell }+iJe_{\ell
})\lrcorner +\frac{1}{2}(\kappa +iJ\kappa )\lrcorner )\} \\
&-\frac{1}{4}(\kappa -iJ\kappa )\wedge (-(e_{\ell }+iJe_{\ell })\lrcorner
\nabla _{e_{\ell }}+\frac{1}{2}(\kappa +iJ\kappa )\lrcorner ) \\
=&-\frac{1}{2}(e_{j}-iJe_{j})\wedge (e_{\ell }+iJe_{\ell })\lrcorner \nabla
_{e_{j}}\nabla _{e_{\ell }}+\frac{1}{4}(e_{j}-iJe_{j})\wedge (\nabla
_{e_{j}}\kappa +iJ(\nabla _{e_{j}}\kappa ))\lrcorner \\
&+\frac{1}{4}(e_{j}-iJe_{j})\wedge (\kappa +iJ\kappa )\lrcorner \nabla
_{e_{j}}+\frac{1}{4}(\kappa -iJ\kappa )\wedge (e_{\ell }+iJe_{\ell
})\lrcorner \nabla _{e_{\ell }} \\
&-\frac{1}{8}(\kappa -iJ\kappa )\wedge (\kappa +iJ\kappa )\lrcorner.
\end{aligned}
$$
Thus, by taking the sum of the last two equations we find
$$
\begin{aligned}
2((\tilde{\bar{\partial}}_{B})^{\ast }\tilde{\bar{\partial}}_{B}+\tilde{\bar{
\partial}}_{B}(\tilde{\bar{\partial}}_{B})^{\ast }) =
\,&\nabla ^{\ast }\nabla
+\frac{1}{4}|\kappa |^{2}+\mathfrak{R}+\frac{1}{4}(e_{j}-iJe_{j})\wedge
(\nabla _{e_{j}}\kappa +iJ(\nabla _{e_{j}}\kappa ))\lrcorner \\
&+\frac{i}{2}(Je_{\ell },\nabla _{e_{\ell }}\kappa )-\frac{1}{4}(\nabla
_{e_{\ell }}\kappa -iJ(\nabla _{e_{\ell }}\kappa ))\wedge (e_{\ell
}+iJe_{\ell }\lrcorner)\,,
\end{aligned}
$$
as required.
\end{proof}
The previous theorem has the following remarkable consequence when it is applied to  $(p,0)$-forms:
\begin{cor} In the hypothesis of theorem $\ref{W}$, for every form $\alpha\in \Lambda_B^{p,0}(M)$ we have
\begin{equation*}
2(\tilde{\bar{\partial}}_{B})^{\ast }\tilde{\bar{\partial}}_{B}\alpha=\nabla
^{\ast }\nabla \alpha+\frac{1}{4}|\kappa |^{2}\alpha+\frac{i}{2}\sum_{j=1}^qR^\nabla(Je_{j},e_{j})\alpha-
\frac{i}{2}{\rm div}_Q(J\kappa)\alpha\,.
\end{equation*}

\end{cor}
Another consequence of theorem \ref{W} is the following 
\begin{theorem}\label{vanishing}
Let $(M,g,\mathcal{F},J)$ be a compact manifold endowed with a K\"ahler
foliation. If the transverse Ricci curvature is negative definite,
then $\mathcal{F}$ has only trivial transversally holomorphic vector fields.
\end{theorem}

\begin{proof}
We show that every  $\xi \in \Lambda_{B}^{1,0}(M)$ satisfying $\bar{\partial}_{B}\xi=0$ is trivial. Such a $\xi$ satisfies $\tilde{\bar{\partial}}_{B}\xi =-\frac{1}{4}(\kappa -iJ\kappa
)\wedge \xi $ and hence $|\tilde{\bar{\partial}}_{B}\xi |_{H}^{2}=
\frac{1}{8}|\kappa |^{2}|\xi |^{2}$. By taking the product with $\xi $
in the Weitzenb\"{o}ck formula and integrating over $M$ we get
\begin{equation*}
\frac{1}{4}\int_{M}|\kappa |^{2}|\xi |_{H}^{2}v_g=\int_{M}|\nabla \xi |^{2}v_g+
\frac{1}{4}\int_{M}|\kappa |^{2}|\xi |_{H}^{2}v_g-\int_{M}H({\rm Ric}^\nabla(\xi ),\xi
)v_{g}.
\end{equation*}
In the above identity, we have used $\frac{1}{2}\sum_{j=1}^qR^\nabla(Je_{j},e_{j})\xi ={\rm Ric}^\nabla(J\xi
)=i{\rm Ric}^\nabla(\xi)$. The assumption on the Ricci curvature to be negative definite implies the statement.
\end{proof}

We point out that theorem \ref{vanishing} was obtained before by S.D. Jung in \cite{jung} by using another method.

\begin{theorem}\label{SUn}
Let $(M,g,\F,J)$ be a compact manifold endowed with a K\"ahler foliation.
If the transverse Ricci curvature vanishes, then every transversally
holomorphic form is parallel. Moreover, if the transverse curvature of $M$ is
positive definite, there every transversally holomorphic $(p,0)$-form on $M$ is trivial.
\end{theorem}

\begin{proof} Let $\gamma$ be a transversally holomorphic $(p,0)$-form.
Then $\bar\partial_B\gamma=0$ and thus $\tilde{\bar \partial}_B\gamma=-\frac{1
}{4}(\kappa-iJ\kappa)\wedge\gamma$ which implies $|\tilde{\bar \partial}
_B\gamma|_H^2=\frac{1}{8}|\kappa|^2|\gamma|^2$. Using the Weitzenb\"ock formula and the identities
$$
R^\nabla(X,Y)\gamma=R^\nabla(X,Y)e_k\wedge
e_k\lrcorner\gamma
$$
and
$$
\frac{1}{2}\sum_{j=1}^q R^\nabla(J(e_i),e_i)e_k={\rm Ric}^\nabla(J(e_k)),
$$
we get the first part of the statement. The second part of the theorem can be obtained following the same proof as in the non-foliated case (see e.g.  \cite{moroianu}).
\end{proof}

\section{Transverse Calabi-Yau structures}\label{sectionSU(n)}
In this short section, we consider {\em transverse Calabi-Yau} structures. Such structures can be defined as Riemannian foliations having transverse holonomy contained in ${\rm SU}(n)$ and have been taken into account in \cite{moro1,moro2} where it is proved that in the {\em taut} case the moduli space is a smooth Hausdorff manifold of finite dimension (i.e. a generalization of the Bogomolov-Tian-Todorov theorem \cite{Bogo, tian, todorov} to the foliated case). Moreover, Calabi-Yau foliations can be used for desingularizing Calabi-Yau orbifolds. 

\medskip
The following proposition is a direct consequence of theorem \ref{SUn}
\begin{prop}\label{transvholomorphic}
On a compact manifold endowed with a Calabi-Yau foliation every transversally holomorphic form is parallel with respect to the transverse Levi-Civita connection.
\end{prop}
Moreover, we have the following:
\begin{prop}\label{SU(n)}
Let $(M,\F,g_Q,J)$ be a  compact simply-connected manifold carrying a  K\"ahler foliation. Assume $\rho_Q=0$; then ${\rm Hol}(\N)$ is contained in ${\rm SU}(n)$ and $\F$ is a Calabi-Yau foliation.
\end{prop}
\begin{proof}
Let $(\F,g_Q,J)$ be a K\"ahler foliation, 
$K:=\Lambda^{n,0}(Q)$ and let $R^K$ be the curvature of the connection induced by the transverse Levi-Civita connection on $K$.  Fix a global section $\psi$ of $K$. Then a standard  computation yields
$$
R^{K}_{X,Y}\psi=\rho_Q(X,Y)\psi
$$
for every pair of smooth vector fields on  $M$. Hence if we assume $\rho_Q=0$, then we have  $R^K=0$ which is equivalent to require that  ${\rm Hol}^0(\N)\subseteq {\rm SU}(n)$.  Hence when $M$ is simply-connected, we have ${\rm Hol}(\N)\subseteq {\rm SU}(n)$, as required. 
\end{proof}
Combining the El Kacimi theorem \ref{transvcalabiyau} with the last proposition, we get the following 
\begin{cor}\label{4.4}
Let  $(\F,g_Q,J)$   K\"ahler foliation  on a compact simply-connected manifold $M$ and assume $c_B^1=0$. Then there exists
a unique transverse K\"ahler form $\omega'$
in the same basic cohomology class of the fundamental form of $g_Q$ whose transverse Levi-Civita connection $\nabla'$ satisfies ${\rm Hol}(\N')\subseteq {\rm SU}(n)$.
\end{cor}

\section{Transverse Hyper-K\"ahler structures}
A particular class of  Calabi-Yau foliations is provided by hyper-K\"ahler foliations. These latters are characterized by a triple $(J_1,J_2,J_3)$ of transverse complex structures on a  foliated manifold  $(M,\F,g_Q)$ such that $(\F,g_Q,J_r)$ is a K\"ahler foliation for every $r=1,2,3$ and it is satisfied the quaternionic relation
\b\label{quaternion}
J_1J_2=-J_2J_1=J_3.
\e
Requiring the existence of this structure on a foliated manifold $(M,\F)$ is equivalent to require  the existence of a transverse metric having transverse holonomy group contained in ${\rm Sp}(n).$ Moreover, if $(\F,g_Q,J_1)$ is a K\"ahler foliation, then it is {\em transversally hyper-K\"ahler} if and only if there exists a basic $(2,0)$-form $\Omega$ such that
$$
\Omega^n\neq 0\,,\qquad \N\Omega=0\,.
$$
If $(J_1,J_2,J_3)$ is simply a triple of transverse complex structures compatible with a fixed transverse metric and satisfying \eqref{quaternion},  we refer to $(g_Q,J_1,J_2,J_3)$ as a {\em transverse  hyper-Hermitian} structure. The main result of this section is the following theorem which is the foliated counterpart of the main theorem
in \cite{verb}

\begin{theorem}\label{mainluigi}
Let $(M,\F,g_Q,J_1)$ be a compact simply-connected manifold carrying a K\"ahler foliation. Assume that there exists a pair $(J_2,J_3)$ of transverse complex structures such that $(J_1,J_2,J_3)$ is a transverse hyper-Hermitian structure.  Then there exists a  transverse metric $g'_Q$ on $(M,\F)$ having holonomy contained in ${\rm Sp}(n)$. 
\end{theorem}

We divide the proof in a sequence of lemmas. The first one is a generalization of theorem 8.2 of \cite{FI} to the foliated non-contact case:

\begin{lemma}\label{ivanov}
For every Hermitian foliation $(\F,g_Q,J)$ on a manifold $M,$ there exists a unique connection $\widetilde{\nabla}$ on $Q$ preserving $(g_Q,J)$ and such that
\begin{enumerate}
\item[1.] $g_Q(T(X_Q,Y_Q),Z_Q)=-g(T(X_Q,Z_Q),Y_Q)$ for every $X,Y,Z\in \Gamma(TM);$

\vspace{0.2cm}
\item[2.] $\widetilde{\nabla}_{\xi}s=\N_{\xi}s$ for every $s\in \Gamma(Q)$ and $\xi\in \Gamma(L)$,
\end{enumerate}
where $T$ is the transverse torsion tensor
$$
T(X,Y)=\widetilde{\nabla}_XY_Q-\widetilde{\nabla}_{Y}X_Q-[X,Y]_{Q}\,.
$$

\end{lemma}
\begin{proof}
Let $g$  be a bundle-like metric on $M$ inducing $g_Q$ and we identify $Q$ with orthogonal complement of $L$ in $TM$. We define explicitly $\tilde \nabla$ as 
\begin{equation}\label{N'}
g(\widetilde{\nabla}_{Z}X,Y)=\begin{cases}
\begin{array}{lcl}
g_Q(\nabla_{Z}X,Y)+\frac12 d\omega(JZ,JX,JY)& \mbox{if} &Z\in\Gamma(Q),\\
g_Q(\nabla_{Z}X,Y)& \mbox{if}& Z\in \Gamma(L)\,,
\end{array}
\end{cases}
\end{equation}
where $\omega$ is the fundamental form of $(g_Q,J)$. First of all, we observe that the connection $\widetilde{\nabla}$ described by formula \eqref{N'} satisfies the two labels of the statement (here we use that $g_Q(T(X,Y),Z)=-d\omega(JX,JY,JZ)).$ On the other hand, the new connection satisfies $\widetilde{\N}g_Q=0,$ since $\N$ preserves $g_Q$ and the second formula in \eqref{N'} forces $\widetilde{\nabla}_{\xi}J=0$ for $\xi \in \Gamma(L)$. Moreover if  $X,Y,Z$ lies in  $\Gamma(Q)$, a direct computation gives
\begin{equation}\label{domega}
2g_Q((\N_{Z}J)X,Y)=d\omega(X,JY,JZ)+d\omega(JX,Y,JZ)
\end{equation}
which implies that $\widetilde{\nabla}$ preserves $J$. This implies the first part of the proof.

\smallskip
We shall now prove the uniqueness. Assume to have a connection $\widetilde{\nabla}$ preserving $(g_Q,J)$ and satisfying the conditions 1 and 2 of the statement.  Therefore, we can write
\begin{equation}\label{newconnection}
g_Q(\widetilde{\nabla}_{Z}X,Y)=g_Q(\N_{Z}X,Y)+\frac12 S(Z,X,Y)
\end{equation}
for a tensor $S$ defined in $TM\times Q\times Q$. Since $\N$ and $\widetilde{\nabla}$ both preserve  $g_Q$, the tensor $S$ is skew-symmetric in the last two entries, i.e.
$$
S(Z,X,Y)=-S(Z,Y,X)\,.
$$The item 2 implies $S(\xi,X,Y)=0$ for all $\xi\in\Gamma(L)$. An easy computation involving the torsion of $\widetilde{\nabla}$, formula  \eqref{newconnection} and the item 1 implies
$$
S(X,Y,Z)=S(Y,Z,X)
$$
for all $X,Y,Z\in\Gamma(Q)$.
Using that $\widetilde{\nabla}$ preserves the structure $J$, it is easy to show that the following relation holds
\begin{equation}\label{Sd}
S(Z,JX,Y)+S(Z,X,JY)=-2g_Q((\nabla_Z J)X,Y)
\end{equation}
for $X,Y,Z\in \Gamma(Q)$. By considering  the cyclic sum in \eqref{Sd} we get
$$
\mathfrak{S}_{Z,X,Y}S(Z,JX,Y)=-\mathfrak{S}_{Z,X,Y}g_Q((\nabla_Z J)X,Y)=-d\omega(Z,X,Y)
$$
and
$$
\mathfrak{S}_{Z,X,Y}S(JZ,X,JY)=d\omega(JZ,JX,JY).
$$
Now the integrability of $J$ gives
$$
S(X, Y, Z)=S(JX,JY, Z)+S(JX, Y, JZ)+S(X, JY , JZ)
$$
which implies
$$
S(X, Y, Z)=d\omega(JZ,JX,JY)\,,
$$
as required.
\end{proof}

\begin{lemma}\label{anticommute}
Let $(\F,g_Q,J_1,J_2,J_3)$ be a transverse hyper-Hermitian foliation and we denote by $\partial_k$ the $\partial_B$ operator with respect to $J_k$. If  $\tilde{\partial}_2:=-J_2^{-1}\bar \partial_1 J_2$, then 
\begin{equation}\label{---}
\partial_1\tilde{\partial}_2=-\tilde{\partial}_2\partial_1\,.
\end{equation}
\end{lemma}
\begin{proof}  Let $g$ be a bundle-like metric on $M$ inducing $g_Q$. We identify $Q$ with the orthogonal complement of $L$ in $TM$. 
The proof of the statement can be  then obtained by using standard algebraic computations. It's
enough to check \eqref{---} for basic maps and basic $(1,0)$-forms. We show how things work for functions and we omit the proof for $1$-forms.

Let $f$ be a basic map and let $Z,W$ be two smooth sections of $Q^{1,0}$. Then we have
$$
\begin{aligned}
(\partial_1\tilde \partial_2f)(Z,W)&=\,Z(\tilde \partial_2f(W))-W(\tilde \partial_2f(Z))- \tilde \partial_2f([Z,W])\\
&=\,ZJ_2W(f)-WJ_2Z(f)-J_2[Z,W](f)
\end{aligned}
$$
and
$$
(\tilde{\partial}_2\partial_1f)(Z,W)=J_2ZW(f)-J_2WZ(f)+J_2[J_2Z,J_2W](f)\,.
$$
Therefore
$$
\begin{aligned}
(\partial_1\tilde \partial_2+\tilde{\partial}_2\partial_1)(f)(Z,W)=&\,\left(ZJ_2W-WJ_2Z-J_2[Z,W]
+J_2ZW-J_2WZ+J_2[J_2Z,J_2W]\right)(f)\\
=&\,J_2N_{J_2}(Z,W)(f)=0\,,
\end{aligned}
$$
as required.
\end{proof}

%

\begin{lemma}\label{fundpre}
Let $(M,\F,g_Q,J_1)$ be a compact  manifold with a taut Calabi-Yau  foliation of codimension $4n$. Assume that there exists a pair $(J_2,J_3)$ of transverse complex structures such that $(J_1,J_2,J_3)$ is a transverse hyper-complex structure. Then  $g_Q$ has transverse holonomy contained in ${\rm Sp}(n)$.
\end{lemma}

\begin{proof}
Let $g_1$ be the transverse Riemannian metric
$$
g_1(\cdot,\cdot)=\frac12 g_Q(\cdot,\cdot)+\frac12 g_Q(J_2\cdot,J_2\cdot)\,.
$$
This metric is compatible with each $J_k$. For each $k=1,2,3,$ we denote by $F_k$ the fundamental form of $(\F,g_1, J_k)$ and let
$$
\Omega_1:=\frac12\left(F_2+i F_3\right).
$$
The basic form $\Omega_1$ is of type $(2,0)$ with respect to $J_1$ and satisfies
$$
\Omega_1^{n}\neq 0\,.
$$
First of all we show that $\Omega_1$ is $\partial_1$-closed, where $\partial_1$ denotes the $\partial_B$-operator with respect to $J_1$. Let $\Omega_2$ the $(2,0)$-component of $\omega$ with respect to $J_2$. The form $\Omega_2$ is basic and from the closure of $\omega$ we get
$$
\partial_2\Omega_2=0,
$$
where $\partial_2$ denotes
the $\partial_B$-operator computed with respect to $J_2$. On the other hand,
an easy computation yields 
$$
\Omega_2=\frac{1}{2}(F_1-i F_3)\,.
$$
Taking into account the formula
\begin{equation*}
\partial_k\gamma=\frac{1}{2}\big(d+(-1)^r i J_k dJ_k\big)\gamma
\end{equation*}
holding for every complex basic form $\gamma$ of degree $r$ (see e.g.  \cite{GP} for the non-foliated case), we deduce that condition $\partial_2\Omega_2=0$ forces to have
$$
J_1\,dF_1=J_3\,dF_3\,.
$$
Let $\tilde \N^k$ be the connection with skew-symmetric torsion induced by $(g_1,J_k)$ as in lemma \ref{ivanov}. Formula \eqref{N'} implies that all the $\tilde \N^k$'s have the same torsion and then we find $\tilde\N^1=\tilde\N^3$ and condition $J_2=J_3J_1$ implies $\tilde\N^1J_2=0$. Therefore
$$
\tilde\N^1=\tilde\N^2=\tilde\N^3
$$
which gives $J_2dF_2=J_3 dF_3$ and thus we obtain that $\partial_1 \Omega_1=0.$ Hence $\Omega_1$ satisfies
$$
\Omega_1^n\neq 0\,,\quad \partial_1\Omega_1=0\,.
$$

\noindent Since $M$ is compact and $\F$ is taut 
we can use the transverse Hodge theory for writing 
$$
\Omega_1=\Omega+\partial_1 \alpha,
$$
where $\Omega$ is the $g_Q$-basic harmonic component of $\Omega_1$. Since $\Omega$ is transversally holomorphic, theorem \ref{SUn} implies that $\N\Omega_1=0$. In particular the norm of $\Omega^n$ is constant. In order to finish the proof, we have to show that $\Omega^n_p\neq 0$ for every $p\in M$. Assume by contradiction $\Omega_p^n=0$ at a point $p\in M$ and let $\psi=\Omega_1^n$. Since the norm of $\Omega^n$ is constant,  $\Omega^n\equiv 0$ and
$$
\psi=\partial_1 \beta
$$
for a basic form $\beta$.
The last step consists in showing that $\psi$ cannot be $\partial_1$-exact. We will obtain this result by adapting the last step in the proof of the main theorem in \cite{verb} to our case.
Taking into account the isomorphism  $\Lambda_B^{2n,1}(M)\cong\Lambda_B^{2n,0}(M)\wedge \Lambda_B^{0,1}(M)$, there exists a basic $(1,0)$-form $\theta$ which is $\partial_1$-closed such that
\begin{equation}\label{verb}
\partial_1 \bar \psi=\theta  \wedge \bar \psi\,.
\end{equation}
Let us consider the complex
\begin{equation}
\label{como}
\Lambda^{0,0}_B(M)\xrightarrow{\partial_1+\frac12 \theta}\Lambda^{1,0}_B(M)\xrightarrow{\partial_1+ \frac12\theta}\Lambda^{2,0}_B(M)\xrightarrow{\partial_1+\frac12\theta}\cdots
\end{equation}
where
$$
\left(\partial_1+\frac12 \theta\right)\alpha=\partial_1\alpha+\frac12\theta\wedge\alpha\,.
$$
Since $\theta$ is $\partial_1$-closed, we have $\left(\partial_1+\frac12\theta\right)^2=0$. In view of \cite[thm. 2.8.7]{EKA} the cohomology of the complex \eqref{como} is finite-dimensional and its cohomology groups can be identified  with the kernel of the Laplacian operator associated with $D_1=\partial_1+\frac12\theta$.
Now, following the approach of \cite{verb2}, we observe that the operator
$$
L\colon \Lambda^{*,0}_B(M)\to \Lambda^{*+2,0}_B(M)
$$
defined as $\beta\mapsto \beta\wedge \Omega_1$
preserves the kernel of  $D_1D_1^*+D_1^*D_1$,
where
$$
D_1^*(\beta):=\partial_1^*\beta+\frac12 *_B(\bar \theta\wedge*_B\beta)=-*_B\bar\partial_1(*_B\beta)+\frac12 \bar\theta\lrcorner\beta
$$
is the formal adjoint to $D_1$ with respect the scalar product \eqref{scalarbasic}  induced by $g_1$.
This basically comes from the following three identities which will be proved afterwards
\begin{eqnarray}
\label{VV3}&& LD_1-D_1L=0\,,\\
\label{VV1}&& D_1D_2+ D_2D_1=0\,,\\
\label{VV2}&& L D_1^*- D_1^* L=\,D_2\,,
\end{eqnarray}
where
$$
D_2(\beta)=-J_2^{-1}\bar \partial_1J_2(\beta)+\frac12 J_2(\bar \theta)\wedge \beta=\tilde\partial_2(\beta)+\frac12 J_2(\bar  \theta)\wedge \beta\,.
$$
Indeed assuming that equalities \eqref{VV3}, \eqref{VV1}  and \eqref{VV2} hold, we write
$$
\begin{aligned}
L(D_1D_1^*+D_1^*D_1)=&\,L D_1D_1^*+LD_1^*D_1=D_1 L D_1^*+D_1^*LD_1+D_2D_1\\
=&\,D_1 D_1^*L+D_1D_2+D_1^*D_1L+D_2D_1=(D_1D_1^*+D_1^*D_1)L\,,
\end{aligned}
$$
as required.
In particular the map $[f]\mapsto [f\psi]$ induces an isomorphism
$$
\ker\, D_1\cap \Lambda^{0,0}_B(M)\to \frac{\Lambda_B^{2n,0}(M)}{D_1(\Lambda^{2n-1,0}_B(M))}.
$$

Assume by contradiction $\psi=\partial_1\beta$. We can certainly find a nowhere vanishing closed basic $(n,0)$-form $\eta$ such that
$$
\bar \psi=g\bar \eta
$$
for a nowhere vanishing basic map $g$. Then
$$
\theta \wedge \bar \psi=\frac{1}{g}\partial_1(g)\wedge \bar \psi\,,
$$
i.e. $\partial_1g=g\theta$. Let $f=g^{-\frac{1}{2}}$; then $f$ gives a non-trivial cohomology class in
$\ker D_1\cap \Lambda^{0,0}_B(M)$. Moreover, $[f\psi]=0$, since
$$
D_1(f\beta)=f\psi\,.
$$
Hence $\psi$ cannot be $\partial_1$-exact and the claim follows.


\medskip 
In order to finish the proof it remains to show that formulas \eqref{VV3}, \eqref{VV1}, \eqref{VV2} are true. For the first one, we simply  have
$$
D_1L\gamma= \partial_1(\gamma\wedge \Omega_1)+ \frac12 \theta\wedge \gamma\wedge \Omega_1 =
\partial_1 \gamma\wedge\Omega_1+\left(\frac12 \theta\wedge  \gamma \right)\wedge \Omega_1=LD_1\gamma
$$
for every $\gamma\in \Lambda^{k,0}_B(M) $. The proof of the other two formulas is a bit more involute and we need lemma \ref{anticommute} and some computations of linear algebra proved as in \cite{verb}. 
First of all, it is easy to show that
$$
J_2\bar\psi=\psi\,.
$$
Therefore $\bar \partial_1(J_2\bar \psi)=\bar \theta \wedge \psi$ which implies the useful relation
$$
\tilde \partial_2 \bar\psi= J_2\bar \theta\wedge \bar\psi\,,
$$
where $\tilde \partial_2 $ is defined in lemma \ref{anticommute}. 
Now applying \eqref{---} to $\bar\psi$ we easily get formula \eqref{VV1}. Formula \eqref{VV2} is equivalent to
\begin{equation}
\big(L D_1^*- D_1^* L)\beta-\frac12\,J_2(\bar\theta)\wedge \beta=\tilde \partial_2(\beta),
\end{equation}
which it is enough to be checked for smooth basic maps and $\tilde \partial_2$-closed $(1,0)$-forms. Let $f$ be a smooth basic map, then
$$
LD_1^*f=0
$$
and
$$
\begin{aligned}
-D_1^*Lf-\frac12 fJ_2\bar\theta=&\,*_B(\bar \partial_1(*_Bf\Omega_1))-\frac12 f\bar\theta\lrcorner \Omega_1-\frac12 fJ_2\bar\theta\\
=&\,*_B\left( \bar \partial_1(f)\wedge (*_B\Omega_1)+ f \bar \theta\wedge *_B\Omega_1\right)-fJ_2\bar\theta=J_2(\bar\partial_1(f))
=\,\tilde\partial_2f
\end{aligned}
$$
where we have used that $*_B\Omega_1=n(\frac{1}{n!})^2\Omega_1^n\wedge\bar\Omega_1^{n-1}$ and the natural identity 
$$
*_B(\bar Z \wedge*_B\Omega_1)=\bar Z \lrcorner \Omega_1=J_2\bar Z.
$$
Now we prove \eqref{VV2} for a basic $\tilde\partial_2$-closed $(1,0)$-form $\alpha$.
We have
$$
LD_1^*\alpha= \left(\partial_1^{*} \alpha\right)\,\Omega_1+\frac12 *_B(\bar \theta\wedge *_B\alpha)\,\Omega_1=\left(\partial_1^{*} \alpha\right)\Omega_1+\frac12 g_1(\bar\theta,\alpha)\Omega_1\,.
$$
$$
\begin{aligned}
-D_1^*L\alpha=&-\partial_1^{*} (\alpha\wedge \Omega_1)-\frac12*_B\left( \bar \theta\wedge *_B(\alpha\wedge \Omega_1)\right)=-\partial_1^{*} (\alpha\wedge \Omega_1)-\frac12\bar\theta\lrcorner(\alpha\wedge\Omega_1)\,.
\end{aligned}
$$

$$
D_2\alpha=\frac12 J_2(\bar  \theta)\wedge \alpha.
$$
Now taking into account that $\alpha$ is $\tilde\partial_2$-closed, we also have
$$
\begin{cases}
 \left(\partial_1^{*} \alpha\right)\,\Omega_1=-*_B(\bar \theta\wedge *_B\alpha)\,\Omega_1\,,\\
 \partial_1^{*} (\alpha\wedge \Omega_1)=-*_B\left( \bar \theta\wedge *_B(\alpha\wedge \Omega_1)\right)
\end{cases}
$$
and therefore 

$$
\begin{aligned}
\left(L D_1^*- D_1^* L\right)\alpha=&\,-\frac12 *_B(\bar \theta\wedge *_B\alpha)+\frac12 *_B\left( \bar \theta\wedge *_B(\alpha\wedge \Omega_1)\right) \\
=&\,\frac12 \bar\theta\lrcorner(\alpha\wedge\Omega_1)-\frac12(\bar\theta\lrcorner\alpha)\Omega_1=-\frac12 \alpha\wedge(\bar\theta\lrcorner\Omega_1)=\frac12 J_2(\bar\theta)\wedge\alpha.
\end{aligned}
$$
i.e.
$$
\left(L D_1^*- D_1^* L\right)\alpha=D_2 \alpha\,,
$$
as required.
\end{proof}

Now we are ready to prove  theorem \ref{mainluigi}
\begin{proof}[Proof of theorem $\ref{mainluigi}$] 
Assume that there exists a pair $(J_2,J_3)$ of transverse complex structures as in the statement. Let 
\begin{equation*}
\Omega(\cdot,\cdot):=\frac12\left(g_Q(J_2\cdot,\cdot)+ig_Q(J_3\cdot,\cdot)\right).
\end{equation*}
This form can be regarded as a basic form of type $(2,0)$ with respect to $J_1$; therefore $\psi:=\Omega^{n}$ is a nowhere vanishing {\em basic}  $(2n,0)$-form and remark \ref{remarkCY} implies 
that the first basic Chern class of $(M,\F,J_1)$ vanishes. 
Since $M$ is simply-connected, the foliation is taut and theorem  \ref{transvcalabiyau} ensures the existence of a transverse metric $g'_Q$ whose transverse Levi-Civita connection $\nabla'$ has holonomy contained in ${\rm SU}(2n)$. 
Hence $(M,\F,g_Q',J_1,J_2,J_3)$ satisfies the hypothesis of lemma \ref{fundpre} and $g'_Q$ has transverse holonomy group contained in ${\rm Sp}(n)$.  
\end{proof}


\medskip
We remark that in the non-foliated case the statement of theorem \ref{mainluigi} holds without the assumption on $M$ to be simply-connected (see \cite{verb}). Indeed, every compact K\"ahler manifold with vanishing first Chern class can be covered by a
K\"ahler manifold with holomorphically trivial canonical bundle and this fact allows us to  drop the  assumption on $M$ to be simply-connected. Unfortunately, it seems that a similar construction cannot be performed in the foliated case. 

\medskip
Examples of hyper-K\"ahler foliations are provided by submersions in hyper-K\"ahler manifolds. Other examples are given in the next section in the set-up of Sasakian manifolds. Moreover, a naif way for constructing hyper-K\"ahler foliations consists in generalizing proposition \ref{coisotropic} to the hyper-K\"ahler case. Indeed, let $(M,g,J_1,J_2,J_3)$ be a hyper-K\"ahler manifold with induced fundamental forms $(\omega_1,\omega_2,\omega_3)$ and let $i\colon N\hookrightarrow M$ be a submanifold. Assume that $N$ is coisotropic with respect to both $\omega_1$ and $\omega_2$ and
$$
(TN)^{\omega_1}=(TN)^{\omega_2}\,.
$$
In this case $\F_p:=(T_pN)^{\omega_1}$  induces a locally trivial hyper-K\"ahler foliation on $N$.

\section{Sasakian manifolds with special transverse holonomy}
In this section, we adapt the results of the previous sections to the Sasakian case having special transverse holonomy. In addition to the previous part of the paper, we consider also the case of a foliated manifold with trivial transverse holonomy group.

\subsection{Riemannian foliations with $\rm{ Hol}(\N)=0$} We have the following result:
\begin{theorem} Let $(M,\mathcal{F},g_Q)$ be a compact manifold endowed with a Riemannian foliation of codimension $q$. The transversal holonomy group is trivial if and only if  $M$ is the total space of a fibration over the flat torus $\mathbb{T}^q.$
\end{theorem}
\begin{proof}
Since the transverse  holonomy group is trivial, there exists a global parallel orthonormal frame $\{s_i\}_{i=1,\cdots,q}$ of sections of $Q$. Thus the $1$-forms $\omega_i$ dual to $s_i$ regarded as $1$-forms on $M$ are  $d$-closed and linearly independent at each point. That means the foliation $\mathcal{F}$ is an $\R^q$-Lie foliation and hence in view of \cite[p.154]{Go}  $M$ is the total space of a fibration over the flat torus.
\end{proof}

\smallskip
Now we consider Sasakian manifolds with trivial transverse holonomy. The key pattern is the following:\\
\noindent Let $\mathfrak{h}_{2n+1}$ be the $2n+1$-dimensional Heisenberg Lie algebra  whose structure equations are given by the choose of a cobasis $\{e^i\}$ satisfying
$$
\begin{cases}
& de^k=0\,,\quad k=1,\dots, 2n\,,\\
& de^{2n+1}=e^1\wedge e^2+e^{3}\wedge e^4+\dots e^{2n-1}\wedge e^{2n}\,.
\end{cases}
$$
(Shortly $\mathfrak{h}_{2n+1}$ has structure equations $(0,\dots,0,12+\dots+(n-1)n)$).
The simply-connected Lie group ${\rm H}$ associated to $\mathfrak{h}_{2n+1}$ has the natural
invariant Sasakian structure
$$
\xi=e_{2n+1}\,,\quad \eta=e^{2n+1}\,,\quad g=\sum e^k\otimes e^k\,,\quad \Phi=e^1\otimes e_2-e^2\otimes e_1+\dots +e^{2n-1}\otimes e_{2n+1}- e^{2n+1}\otimes e_{2n-1}
$$
$\{e_i\}$ is being the dual basis to $\{e^i\}$. It is standard to check that such a Sasakian structure satisfies
${\rm Hol}(\N)=0$. Hence if $\Gamma\subseteq {\rm H}$ is a co-compact lattice, the compact manifold $M=\Gamma \backslash {\rm H}$ inherits a natural $\N$-flat Sasakian structure.

The next result says that every $\N$-flat compact Sasakian manifold is of the form $M=\Gamma\backslash {\rm H}$, for some lattice $\Gamma$.

\begin{theorem}\label{1}
Let $(M,\xi,\eta,\Phi,g)$ be a compact Sasakian manifold. Then the holonomy group of $\N$
is trivial if and only if $M$ is a compact quotient of the odd-dimensional Heisenberg Lie group ${\rm H}$ by a lattice and
$(\xi,\eta,\Phi,g)$ lifts to an invariant Sasakian structure on ${\rm H}$.
\end{theorem}

\begin{proof}
Let $(M,\xi,\eta,\Phi,g)$ be a Sasakian manifold. The holonomy group of  $\N$ is trivial if and only if there exists a global transverse unitary frame $\{Z_r\}$ satisfying
$$
\N_{\ov{Z}_r}Z_k=\N_{Z_r}Z_k=\N_{\xi}Z_k=\N_{\xi}\ov{Z}_k=0\,,\quad r,k=1,\dots,n\,.
$$
Conditions $\N_{Z_r}Z_k=\N_{Z_r}{\ov Z}_k=0$ say that
$$
\N_{Z_r}^gZ_k=0\,,\quad \N_{\ov{Z}_r}^gZ_k=-i\delta_{rk}\,\xi
$$
whilst condition $\N_{\xi}Z_k=0$ can be rewritten in terms of brackets as
$$
[Z_k,\xi]=0\,,\quad k=1,\dots,n\,.
$$
It follows that there exists a frame $\{X_i\}$ on $M$ such that $[X_i,X_j]=\sum \lambda_{ij}^k\, X_k$ for some constants $\lambda_{ij}^k$. In view of \cite{palais}, $M$ can be written as a quotient of a simply-connected nilpotent Lie group $N$ by a lattice.
The vector fields $\{Z_i,\xi\}$ lift to invariant vector fields on $N$. Let $\{\zeta^i,\eta\}$, be the dual frame to
$\{Z_i,\xi\}$; then
$$
d\zeta^i=0\,,\quad d\eta=i\,\sum_k \zeta^{k}\wedge \ov{\zeta}^k
$$
and $N$ is the Heisenberg Lie group ${\rm H}$, as required.
\end{proof}

\begin{rem}{\em
Notice that from the point of view of transverse geometry, manifolds associated to the Heisenberg group play in Sasakian geometry the role that complex tori play in K\"ahler geometry.}
\end{rem}

\subsection{Sasakian manifolds with ${\rm Hol}(\N)\subseteq{\rm SU}(n)$}\label{sakSu(n)}
In \cite{TV} it has been  studied the geometry of Sasakian manifolds satisfying ${\rm Hol}(\nabla)\subseteq {\rm SU}(n)$. These manifolds were named {\em contact Calabi-Yau}.
The main result of \cite{TV} is a generalization  of the McLean theorem (see \cite{Maclean}) to the Sasakian context, where the role of Lagrangian submanifolds was replaced by some special
Legendrian immersions. Indeed, a {\em contact Calabi-Yau} manifold can be defined as a $2n+1$-dimensional Sasakian manifold $(M,\xi,\eta,\Phi,g)$ with an addition structure given by a basic closed transverse complex volume form $\epsilon$. In analogy to the classical Calabi-Yau case, the real part of $\epsilon$ is a calibration on $M$ (see \cite{HL}) whose calibrated submanifolds are given by $n$-dimensional smooth embeddings $i\colon N\hookrightarrow M$ satisfying
$$
\begin{cases}
i^*(\eta)=0\\
i^*(\Im\mathfrak{m}\,\epsilon)=0\,.
\end{cases}
$$
Condition $i^*(\eta)=0$ says that $N$ is a {\em Legendrian} submanifold and for this reason such submanifolds were named in \cite{TV} {\em special Legendrian}.  The main result of \cite{TV} is the following:
\begin{theorem}[Tomassini-Vezzoni \cite{TV}]
The moduli space of compact Legendrian submanifolds isotopic to a fixed one is always a $1$-dimensional smooth manifold.
\end{theorem}
\noindent Some further results about special Legendrian submanifolds with boundary are pointed out in \cite{lucking}.

\medskip
In this section, we observe that links provide examples of simply-connected compact contact Calabi-Yau manifolds. In order to show this, we consider the following two results arising from section \ref{sectionSU(n)} and the Sasakian version of the El Kacimi theorem (see \cite{BGN}):

\begin{prop}\label{SU(n)Sak}
Let $(M,\xi,\eta,\Phi,g)$ be a  compact simply-connected null Sasaki $\eta$-Einstein manifold. Then ${\rm Hol}(\N)$
is contained in ${\rm SU}(n)$.
\end{prop}
Here we recall that in Sasakian geometry {\em null $\eta$-Einstein} means that the transverse Ricci tensor vanishes.
\begin{prop}\label{imp}
Let $(M^{2n+1},\xi,\eta,\Phi,g)$ be a compact simply-connected Sasaki manifold with $c_B^1(\F)=0$. Then there exists a basic $1$-form $\zeta$ on $(M,\xi)$ and a unique Sasakian structure $(M,\xi',\eta',\Phi',g')$ such that
$$
\xi'=\xi\,,\quad \eta'=\eta+\zeta\,,\quad \Phi'=\Phi-\zeta\otimes\xi\circ\Phi\,,\quad g'=d\eta'\circ({\rm Id}\otimes \Phi')+\eta'\otimes \eta'
$$
and the transverse holonomy group of the metric $g'$ is contained in ${\rm SU}(n).$
\end{prop}

It is known that Links  provide examples of simply-connected null Sasakian $\eta$-Einstein manifolds.
More precisely, given a link $L_f=C_{f}\cap S^{2n-1}$, where $f=(f_1,\dots, f_p)$ are independent weighted homogeneous polynomials of degrees
$(d_1,\dots,d_p)$ and weights $(w_1,\dots, w_p)$, then  $L_f$
is $(n-p-1)$-connected (see \cite{L}) and inherits a natural $\eta$-Einstein Sasakian structure $(\xi,\eta,\Phi,g)$ induced by the weighted Sasakian
structure of the sphere.  Since $L_f$ is null whenever $\sum (d_i-w_i)=0$ (see e.g. \cite{BGlibro}), by making use of proposition \ref{imp}, we infer

\begin{prop}\label{links}
Let $L_f$ be a link, where $f=(f_1,\dots, f_p)$ are independent weighted homogeneous polynomials of degrees
$(d_1,\dots,d_p)$ and weights $(w_1,\dots, w_p)$ and assume $\sum (d_i-w_i)=0$. Then $L_f$ carries a contact Calabi-Yau structure.
\end{prop}

\subsubsection{A link with $G_2$-geometry} Let $(M,\xi,\eta,\Phi,g,\epsilon)$ be a $7$-dimensional {\em contact Calabi-Yau manifold} and consider the $3$-form
$$
\sigma=\eta\wedge d\eta+\Re\mathfrak e\,\epsilon.
$$
Then $\sigma$ induced a $G_2$-structure on $M$. Since
$$
d*\sigma= 0
$$
this induced $G_2$-structure is always {\em co-calibrated}. A similar construction can be done in  $7$-dimensional $3$-Sasakian manifolds (see \cite{AF}). Therefore, proposition \ref{links} readily implies
\begin{cor}
Let $L_f$ be a link as in proposition  $\ref{links}$. Then $L_f$ has a co-calibrated $G_2$-structure.
\end{cor}

\subsection{Sasakian manifolds with ${\rm Hol}(\N)\subseteq{\rm Sp}(n)$}In this subsection we translate theorem \ref{mainluigi} to the context of Sasakian manifolds.

\medskip
Condition ${\rm Hol}(\N)\subseteq {\rm Sp} (n)$ for a Sasakian manifold $(M,\xi,\eta,\Phi_1,g)$ is equivalent to require the existence  of a pair $\Phi_2,\Phi_3\in {\rm End} (TM)$ such that
\begin{equation}\label{sas}
(\xi,\eta,\Phi_k,g) \mbox{ is a Sasakian structure for } k=2,3
\end{equation}
and satisfy the transverse quaternionic relations
\begin{equation}\label{quater}
\Phi_1\Phi_2=-\Phi_2\Phi_1\,,\quad  \Phi_1\Phi_2=\Phi_3\,.
\end{equation}
Conditions \eqref{sas} can be alternatively rewritten as
\begin{eqnarray}
&& \label{3} \Phi_k^2=-{\rm I}+\eta\otimes \xi\\
&& \label{4} N_{\Phi_k}=0\\
&& g(\Phi_k \cdot,\Phi_k \cdot)=g(\cdot,\cdot)-\eta(\cdot)\eta(\cdot)
\end{eqnarray}
where $N_{\Phi_k}$ is the Nijenhuis tensor
$$
N_{\Phi_k}(X,Y)=[\Phi_kX,\Phi_kY]-\Phi_k[\Phi_kX,Y]-\Phi_k[X,\Phi_kX]+\Phi_k^2[X,Y]\,.
$$


The following result is the generalization of theorem \ref{mainluigi} to the Sasakian context: 
\begin{theorem}\label{sasaki}
Let $(M,\xi,\eta,\Phi_1,g)$ be a compact simply-connected $4n+1$-dimensional Sasakian manifold.  Assume that  there exists  a pair $\{J_2,J_3\}$ of transverse complex structures inducing with $(g,\Phi_1)$ a transverse hyper-Hermitian structure. Then
there exists a Sasakian structure $(\xi,\eta',\Phi',g')$ on $M$ having transverse holonomy contained in 
${\rm Sp}(n)$.
\end{theorem}

\begin{proof}
The existence of $\{J_2,J_3\}$ implies that  the first basic Chern class of $(M,\xi,\eta,\Phi_1,g)$ is zero. Then in view of proposition  \ref{imp}, there exists a Sasakian structure $\mathcal{S}'=(\xi,\eta',\Phi',g')$ on $M$ 
$$
\eta'=\eta+\zeta\,,\quad \Phi'=\Phi-\xi\otimes\zeta\circ\Phi\,,\quad g'=d\eta'\circ({\rm Id}\otimes \Phi')+\eta'\otimes \eta'
$$
having transverse holonomy contained in ${\rm SU}(2n)$. Lemma \ref{fundpre} implies that the transverse Levi-Civita connection of $g'$ has holonomy contained in ${\rm Sp}(n)$, as required. 
\end{proof}


Note that, since ${\rm Sp}(1)={\rm SU}(2)$,  corollary \ref{imp} implies that in dimension $5$  every simply-connected Sasakian manifolds satisfying ${\rm Hol}(\N)\subseteq {\rm Sp}(1)$  is in fact a compact  simply-connected
null Sasaki $\eta$-Einstein manifolds. These kind of manifolds are classified in \cite{cuadros} where it is showed
that a $5$-dimensional simply-connected compact manifold admits a null Sasaki $\eta$-Einstein structure if and only if
it is obtained as a connected sum  of $k$-copies of
$S^2\times S^3$, where $k=3,\dots,9$.

\end{document}